\documentclass [11pt]{amsart}

\usepackage{amsmath,amsthm, amsfonts, amssymb, setspace}
      \newtheorem{example}{Example}
      
      \newtheorem{theorem}{Theorem}
      \newtheorem{corollary}{Corollary}
      \newtheorem{lemma}{Lemma}
      
      \newtheorem{proposition}{Proposition}
      \newtheorem{assumptions}{Assumption}
      \def\R{{\mathbb R}}
      \def\N{{\mathbb N}}
      \def\C{{\mathbb C}}
      \def\D{{\mathbb D}}

      \def\Mnd{M_n(\mathbb C^d)}
      \def\cN{\mathcal C}
      \def\cFd{\mathcal F_d}
      \def\cFdd'{\mathcal F_{dd'}}
      \def\Cg{\mathbb C\langle g\rangle}
      \def\cU{\mathcal U}
      \def\cAKn{\mathcal A(\mathcal  K)}
      \def\cAK{\cAKn^\infty}
      \def\cIK{\mathcal I(\mathcal  K)}
      \def\cK{\mathcal K}
      \def\cM{\mathcal M}
      \def\cH{\mathcal H}
      \def\cV{\mathcal V}
      \def\cW{\mathcal W}
      
      \def\cA{\mathcal A}
      \def\cD{\mathcal D}

\begin{document}
\title{Non-Commutative Carath\'{e}odory Interpolation}

\author[Balasubramanian]{Sriram Balasubramanian}
\address{Department of Mathematics\\
  Univeristy of Florida }
\email{bsriram@ufl.edu}

\subjclass[2000]{47A57, 47L30 (Primary).
47A13 (Secondary)}

\keywords{Interpolation, Carath\'{e}odory, Carath\'{e}odory-Fej\'{e}r, Abstract operator algebra,
BRS, Matrix convex set, Formal power series}

\maketitle

\begin{abstract}
We prove a Carath\'{e}odory-Fej\'{e}r
 type interpolation theorem for certain
matrix convex sets in $\C^d$ using the
Blecher-Ruan-Sinclair characterization of abstract
operator algebras.  Our results generalize the work of
Dmitry S. Kalyuzhny\u{i}-Verbovetzki\u{i}
for the d-dimensional non-commutative polydisc.
\end{abstract}

\thispagestyle{empty}

\section{Introduction}
A classical interpolation problem in function theory is the
Carath\'{e}odory-Fej\'{e}r interpolation
problem (CFP):
Given $n+1$ complex numbers $c_0, c_1, ..., c_n$ does there exist
a complex
valued analytic function $f(z) = \sum_{j=0}^{\infty}
f_j z^j$
 defined on the open unit disc $\D \subset \C$
 such that $f_j = c_j$ for all $0 \leq j \leq n$ and
 $|f(z)| \leq 1$ for all $z \in \D$?

The problem and some of its variants were studied by Carath$\acute{\text{e}}$odory,
Fej$\acute{\text{e}}$r, Schur and Toeplitz in \cite{Sc}, \cite{T} and
\cite{CF}.
 A necessary and sufficient condition, commonly referred to as
 the Schur criterion, for the existence of a solution
 to the problem is that the matrix
 \begin{equation}
 \label{eq:matrix}
 \begin{pmatrix} c_{0} & 0 & \cdots & 0 \\ c_{1} & \ddots
 & \ddots & \vdots \\ \vdots & \ddots & \ddots & 0 \\
 c_{n} & \cdots & c_1 & c_{0} \end{pmatrix}
 \end{equation}
 is a contraction. A detailed exposition of the CFP
 and its numerous function-theoretic and engineering
 applications can be found in the comprehensive
 book of Foias and Frazho \cite{FF}.  The operator
 theoretic formulation of Sarason \cite{S} has had
 a major impact on the CFP and the related Pick interpolation
 problem and the development of opeator theory and
 the study of non-self adjoint operator algebras  generally.

 From the operator theory/algebra point of view, the CFP
 is essentially unchanged if the coefficients
 $c_0, c_1, ..., c_n$ are taken to be elements of $B(\cU)$ for some separable
 Hilbert space $\cU$. Indeed,
The Schur criterion in this case can be formulated
in the following way.
Let $p(z) = \sum_{j=0}^n c_jz^j : \D \to B(\cU)$ be given.
There exists an analytic function $f:\mathbb D\to B(\cU)$
such that $f^{(j)}(0)=c_j$ for $0\le j\le n$ if and only if
the norm of $p(T)=\sum_{j=0}^n c_j\otimes T^j$
is at most one for every operator $T$ on Hilbert space
which is nilpotent of order $n+1$;  i.e., $T^{n+1}=0$.

Several commutative multi-variable generalizations of the
CFP  have
been obtained for different domains for example,
the polydisc $\D^d \subset \C^d$,
and for different interpolating classes of functions,
for example the Schur-Agler class of analytic
functions that take contractive operator values on
any d-tuple of commuting strict contractions in a manner
discussed in \cite{A}. For more details see \cite{EPP}, \cite{BLTT}.
Some results on the problem for bounded circular domains
in $\C^d$ can also be found in \cite{D}.

The broad purpose of this article is to extend some of the results in
the commutative case to the noncommutative setting of the free algebra
on $d$ generators. Some non-commutative
generalizations of the CFP have
already been studied in \cite{P2}, \cite{P3}, \cite{Co}, \cite{KV},
\cite{BGM}. Here we pose the problem
for domains that are matrix convex sets (see \cite{EW})
 in $\C^d$. The domains considered here include as specific examples,
the $d$-dimensional non-commutative matrix polydisc, which is the domain
used in \cite{KV}, the $d$-dimensional non-commutative
matrix polyball and the $dd'$-dimensional non-commutative matrix mixed ball.
Using non-commutative matrix (operator) valued analytic functions
 on matrix-convex sets -  formal power series
with matrix (operator) coefficients that converge on some non-commutative
neighborhood of 0 (see \cite{V1}, \cite{V2}, \cite{V3}, \cite{P1}, \cite{P2},
 \cite{P3}, \cite{KVV}, \cite{HKMS}) -  and
the Blecher-Ruan-Sinclair characterization of
abstract operator algebras, we prove an interpolation theorem
from which a necessary and sufficient
condition for the existence of a minimum-norm solution to the
CFP follows.

The article is structured as follows: In Section \ref{sec:mconvex},
matrix convex
sets in $\C^d$ and the interpolating class $\cAK$ are introduced and
a basic version of the main result is stated. In Section
 \ref{sec:opalgcak}, it is
established that the interpolating class $\cAK$ is an abstract
operator algebra. In Section \ref{sec:weak}, a weak-compactness property
of the algebra $\cAK$ is proved. Section \ref{sec:quotient}
 begins with the
definition of the ideal $\cIK \subset \cAK$ which
plays a role analogous to that of the ideal of analytic functions
in $H^\infty$ of
the unit disc which vanish to order $n$ at $0$
in the classical CFP. This section also contains the
discussion of why $\cAK/\cIK$ is an abstract operator algebra.
It is also shown in this section that the norms of classses
in the quotient algebra are attained. In Section
\ref{sec:reps}, completely contractive representations of the algebra $\cAK$
are studied and it is also shown that finite-dimensional
compressions of operators
that give rise to completely contactive representations of $\cAK$
 lie on the boundary of the
underlying matrix convex set. Section \ref{sec:opversion}
 contains the matrix and
operator versions of the CFP for (certain) matrix-convex sets in $\C^d$
and the main interpolation theorem. The article ends with some examples
for which a variant of the main interpolation theorem from Section
 \ref{sec:opversion}
holds even in the case of infinite initial segments $\Lambda$.
See Section \ref{sec:further}.

\section{Matrix Convex Sets in $\mathbb C^d$}
 \label{sec:mconvex}
A basic object of study in this article is a {\it quantized}, or
{\it non-commutative}, \index{non-commutative set}\index{quantized} version of
a convex set.  While the definitions easily extend
to convex subsets of arbitrary vector spaces, here
the focus is on subsets  of $\mathbb C^d$.
In this section we review the definition of a matrix convex subset
of $\mathbb C^d$ and introduce our standard assumptions regarding these
sets.

\subsection{Non-commutative sets}
Let $M_{m,n}=M_{m,n}(\mathbb C)$ denote the $m\times n$ matrices over
$\mathbb C$.  In the case that $m=n$, we write $M_n$ instead of
$M_{n,n}$.  Let $M_n(\mathbb C^d)$ denote $d$-tuples with entries from
$M_n$. Thus, an $X\in \Mnd$ has the form $X=(X_1,\dots,X_d)$
where each $X_j\in M_n$.   A {\it non-commutative set}
\index{non-commutative set} $\mathcal L$
is a sequence $(\mathcal L(n))$ where, for each positive
integer $n$,  $\mathcal L(n) \subset \Mnd$ which is
{\it closed with respect to direct sums}; i.e., if
$X\in \mathcal L(n)$ and  $Y\in\mathcal L(m)$, then
\begin{equation}
 \label{eq:direct-sum}
   X\oplus Y = (X_1\oplus Y_1,\dots, X_d\oplus Y_d)
    \in \mathcal L(n+m)
\end{equation}
 where
\[
  X_j\oplus Y_j =  \begin{pmatrix} X_j & 0 \\ 0 & Y_j \end{pmatrix}.
\]

  A non-commutative set $\mathcal L=(\mathcal L(n))$ is
  {\it open} if each $\mathcal L(n)$ is open.
  \index{open, non-commutative set}

\subsection{Convexity}
A {\it matrix convex set $\cK=(\cK(n))$}
\index{matrix convex set} is a non-commutative set which is
{\it closed with respect to conjugation by an isometry};
\index{closed with respect to conjugation} i.e.,
if  $\alpha \in M_{m,n}$ and  $\alpha^* \alpha = I_n$, and if
    $X=(X_1,\dots,X_d)\in \cK(m),$ then
\begin{equation}
 \label{eq:conjugate}
    \alpha^* X\alpha =(\alpha^* X_1\alpha,\dots,\alpha^*X_d\alpha)
       \in \cK(n).
\end{equation}
  Note that, by choosing $n=m$ and $\alpha$ a unitary matrix,
  the condition \eqref{eq:conjugate} implies that each $\mathcal L(n)$
  is closed with respect to unitary conjugation.

  It is a simple matter to combine
  conditions \eqref{eq:direct-sum} and \eqref{eq:conjugate}
  to conclude, if $\mathcal L$ is matrix convex,
  then each $\mathcal L(n)$ is itself convex.

 \subsection{Circled domains}
 A subset  $\cU$ of  $\Mnd$ is {\it circled}  \index{circled}
 if $e^{i \theta} \cU \subseteq \cU$ for all $\theta \in \R$.
 A matrix convex set $\cK$ is circled if each $\cK(n)$ is
 circled.  As a canonical example of a circled matrix convex set,
 suppose $\gamma>0$ and consider the non-commutative
 $\gamma$-neighborhood  $N_\gamma = (\cN_\gamma(n))$ of
 $0 \in \mathbb C^d$ defined by
\[
  \cN_\gamma(n)=\{X\in\Mnd: \displaystyle \sum_{j=1}^d X_jX_j^* < \gamma^2 \}.
\]
 For a matrix convex set $\cK,$ unless otherwise noted,
 it is assumed there exist $\gamma, \Gamma >0$ such that
\begin{equation}
 \label{eq:above-below}
   N_\gamma \subseteq \cK \subseteq N_\Gamma,
\end{equation}
 where the inclusions are interpreted termwise.
 Equivalently, $\cK$ is bounded (contained in some
 non-commutative neighborhood of $0$) and contains a
 non-commutative neighborhood of $0$.

\begin{assumptions}
 \label{assume}
  In this article,  it is typically assumed that $\cK$
\begin{itemize}
 \item[(a)] is open;
 \item[(b)] is bounded;
 \item[(c)] is circled;
 \item[(d)] is matrix convex; and
 \item[(e)] contains a non-commutative neighborhood of $0.$
\end{itemize}
\end{assumptions}

\subsection{Examples of Matrix Convex Sets}
\label{subsec:examples}
 At this point we pause to consider some further examples
 of open, bounded matrix convex sets which contain a
 non-commutative neighborhood of $0$.

\begin{example}
\label{eg:ncmatrixpolydisc}
 Let $\cK(n) = \left\{(X_1, X_2, ..., X_d) \, :
 \, X_j \in M_n \,\, \& \,\, \|X_j\| < 1 \right\}$
with $\gamma < \frac{1}{\sqrt{d}}$ and $\Gamma > \sqrt{d}$.
$\cK= (\cK(n))$ is the d-dimensional noncommutative matrix polydisc.
\end{example}

\begin{example}
\label{eg:ncmatrixmixedball}
Let $\cK(n) = \{X = (X_{1,1}, X_{1,2}, ..., X_{d,d'}) \, : \,
 X_{i,j} \in M_n\,\, \& \,\, \| X \|_{op} < 1 \}$, where
$\|X\|_{op}$ is the norm of the operator $(X_{ij})_{i, j = 1}^{d, d'}
: (\C^n)^{d'} \to (\C^n)^d$ with $\gamma < \frac{1}{\sqrt{dd'}}$
and $\Gamma > \sqrt{dd'}$ is the $d\times d'$ non-commutative
matrix mixed ball.
\end{example}

\subsection{The Interpolating class $\cAK$}
  Let $\mathcal K$ denote a matrix convex set satisfying
  the conditions of Assumption {\ref{assume}}.
  A central object of this article is an algebra of
  formal power series in non-commuting variables which
  converge uniformly on the matrix convex set $\mathcal K.$
  These power series are defined in terms of the
  free semi-group on $d$ letters.

\subsubsection{The Free Semi-group on $d$ Letters}
 Let $\cFd$ denote the free semigroup of words generated by $d$
 symbols ${g_{1},..., g_{d}}$.  The product on $\cFd$ is
 defined by concatenation.
 Thus, if $w = g_{i_1}...g_{i_m}$ and $w' = g_{j_1}...g_{j_n}$, then
 the product $ww'$ is given by
 $g_{i_1}...g_{i_m}g_{j_1}...g_{j_n}$.
 The empty word $\emptyset $ acts as the identity so that
 $w\emptyset=w=\emptyset w$.
 The length of the word $w = g_{i_1}... g_{i_m}$ is $m$
  and is denoted $|w|$. The length of $\emptyset$ is zero.

\subsubsection{The Set $\mathcal{A}^{\infty}$ of Formal Power Series}
 \label{subsec:formal-pow}
   A {\it formal power series} \index{formal power series}
   with entries from $\C$  is an expression of the form
 \begin{equation}
  \label{eq:formal-series}
    \displaystyle\sum_{w \in \cFd}{f_{w} w}
 \end{equation}
    where $f_{w}\in\mathbb C$ (more general coefficients
  $f_w$ will be considered later).
   It is convenient to sum $f$ according to its
   homogeneous of degree $j$ terms; i.e.,
 \begin{equation}
  \label{eq:formal-series-by-deg}
    f=\sum_{j=0}^\infty \sum_{|w|=j} f_w w = \displaystyle
    \sum_{j=0}^\infty f_j.
 \end{equation}

   Give
  a $d$-tuple $T=(T_1,\dots,T_d)$ of operators on a common
   Hilbert space $\mathcal H$ and a word
  $w = g_{i_{1}}g_{i_{2}}...g_{i_{k}} \in \cFd$,
  $i_{1},...,i_{k} \in \{1, 2,...,d\}$,
  define the evaluation of $w$ at $T$
  by
\[
  T^{w}=T_{i_{1}}T_{i_{2}}...T_{i_{k}}.
\]

  Given a formal power series $f$ as above, define
\begin{equation}
  \label{eq:f(T)}
     f(T) = \displaystyle\sum_{j = 0}^{\infty}
       \sum_{|w| = j}{f_{w} T^{w}}
\end{equation}
   provided the sum converges in the operator norm
   in the indicated order.  Convergence in norm is not terribly
   important here and it is possible to use
   instead convergence in the strong or weak operator topologies
   for instance, \cite{BGM}.

  Fix now a matrix convex set $\mathcal K=(\mathcal K(n))$ which
  satisfies the conditions of Assumption {\ref{assume}}. We will
  write $X \in \cK$ to denote $X \in \bigcup_{n \in \N} \cK(n)$.
   Let
\begin{equation*}
  \cAK= \left\{f = \displaystyle\sum_{w \in F_d}f_w w : \, f_w\in\mathbb C,
  \mbox{and for every } X \in \cK, \,
  f(X) \, \text{ converges }\right\}.
\end{equation*}

   For $f \in \cAK$,
   define
 \[
  \|f\| = \sup \{\|f(X)\| \, : \, X \in \cK\}.
\]
   Of course, as it stands, this supremum can be infinite.
   Let
\begin{equation*}
 \cAK = \left\{f = \displaystyle\sum_{w \in F_d}f_ww : f \in \cAK, \|f\| < \infty \right\}.
\end{equation*}
   Thus, elements of $\cAK$ are in some sense analogous to
   $H^\infty$ functions on the unit disk $\mathbb D$.
   It is not hard to see that
  $\|\cdot\|$ is in fact a norm on $\cAK$ and not just
  a semi-norm.

  It will be shown that $\cAK$ is an algebra and  it will also be necessary -
  and desirable - to consider formal power series with matrix and
   operator-valued
  coefficients.  Discsussion of these topics is postponed until
  after stating a base version of the main result of this paper.

\subsection{The Main Result}
  A set $\Lambda\subset \cFd$ is an {\it initial segment}
 \index{initial segement} if
 its complement is an ideal in the semi-group $\cFd$; i.e., if both
 $g_{j}w, wg_j \in \cFd\setminus \Lambda$ $ (1 \leq j \leq d)$, whenever
 $w \in \cFd\setminus \Lambda$. In the case that $d=1$ an initial
 segment is thus a set  of the form $\{\emptyset,g_1,g_1^2,\dots,g_1^m\}$
 for some $m$.

 A tuple $X\in \cK$ is {\it $\Lambda$-nilpotent}
 \index{$\Lambda$-nilpotent}
 provided $X^v=0$ whenever $v\in\cFd\setminus \Lambda$.
 If $\Lambda$ is a finite initial segment,
 $X\in\cK$ is $\Lambda$-nilpotent,  and
 $f$ is as in equation \eqref{eq:formal-series-by-deg},
 then
\[
  f(X)=\sum_{w\in\Lambda} f_w X^w.
\]

\begin{theorem}
 \label{thm:main}
  Fix a matrix convex set $\cK$ satisfying
  the conditions of Assumption {\ref{assume}}.
  Let $\Lambda$, a finite  initial segment, and
\[
  p=\sum_{w\in\Lambda} p_w w
\]
  be given. There exists $f\in\cAK$ such
  that $f_w=p_w$ for $w\in\Lambda$ and
  $\|f\|\le 1$ if and only if
\[
  \sup\{\| p(X)\| :
   X\in\mathcal K,\ \ X\mbox{ is } \Lambda-\mbox{nilpotent}\}\le 1.
\]
\end{theorem}

 The body of the paper
 contains a more general version of Theorem {\ref{thm:main}}
 allowing for operator-valued coefficients $p_w$ and $f_w$.
 A version of the Theorem for the case of infinite initial
  segments $\Lambda$ is also presented in Section {\ref{sec:further}}
  for some specific noncommutative domains.

\section{The operator algebra $\cAK$}
 \label{sec:opalgcak}
  Broadly speaking, the strategy for proving
  Theorem {\ref{thm:main}} is to realize $\cAK$
  as an operator algebra, note that $\Lambda$
  determines a closed ideal in $\cAK$ and then apply
  the important corollary of the BRS theorem
 (See \cite{P}) which says that the quotient of
 an operator algebra by a closed (two-sided) ideal
 has a completely isometric representation
 as a subalgebra of the bounded operators on some Hilbert space.

  The norm on $\cAK$ defined in  subsection {\ref{subsec:formal-pow}}
  naturally generalizes to $m\times n$ matrices
  with entries from $\cAK$ and the resulting
  sequence of norms makes $\cAK$ an abstract
  operator algebra. This section contains the details
  of the construction beginning with proving that $\cAK$ is an algebra.

\subsection{The Noncommutative Fock Space}
 Let $\Cg=\mathbb C\langle g_1,\dots,g_d\rangle$
 denote the algebra of  non-commuting polynomials
 \index{non-commutative polynomial} in the
 variables $\{g_1,\dots,g_d\}.$ Thus elements
 of $\Cg$ are linear combinations of elements of $\cFd$; i.e.,
 an element of $\Cg$ of degree (at most) $k$
 has the form
\[
  \sum_{j = 0}^k \sum_{|w|= j} {p_{w} w},
\]
 where the $p_w$ are complex numbers.

 To construct the Fock space, $\mathbb F^2$, define an inner product
 on $\Cg$ by defining
\begin{equation*}
  \langle w, v \rangle = \begin{cases} 0 & \text{if }w \neq v\\
                                       1 & \text{if } w = v
     \end{cases}
\end{equation*}
  for $w,v\in\cFd$ and extending by linearity to all of $\Cg$.
  The completion of $\Cg$ in this inner product is then
  the Hilbert space $\mathbb F^2$.

\subsection{The Creation Operators}
 \label{subsec:genesis}
  There are natural isometric operators on $\mathbb F^2$ called
  the creation operators which have been studied intensely in
  part because of their connection to the Cuntz algebra \cite{C}.
  Given $1\le j\le d$, define $S_j:\mathbb F^2\to \mathbb F^2$
  by $S_j v=g_jv$ for a word $v\in\cFd$ and extend $S_j$
  by linearity to all of $\Cg$.  It is readily verified that
  $S_j$ is an isometric mapping of $\Cg$ into itself and it
  thus follows that $S_j$ extends to an isometry on all
  of $\mathbb F^2$. In particular $S_j^* S_j = I$, the identity on $\mathbb F^2$.
  Also of note is the identity,
\begin{equation}
 \label{eq:co-isometry}
   \displaystyle \sum_{j=1}^d S_j S_j^* = P,
\end{equation}
  where $P$ is the projection onto the orthogonal complement of the one-dimensional
  subspace of $\mathbb F^2$ spanned by $\emptyset$, which follows
  by observing, for a word $w\in\cFd$ and $1\le j \le d$, that
\[
  S_{j}^*(w)  = {\begin{cases} w' & {\text{if $w = g_{j}w'$}}\\
                                       0 & {\text{otherwise.}}\end{cases}}
\]

 Of course, as it stands the tuple $S=(S_1,\dots,S_d)$ acts
 on the infinite dimensional Hilbert space $\mathbb F^2$.
 There are however, finite dimensional subspaces which are
 essentially determined by ideals in $\cFd$ and which are
 invariant for each $S_j^*$.

  The subset
  $\Lambda(\ell) = \{w \,:\, |w| \leq \ell\}$ of $\cFd$
  is a canonical example of a finite initial segment.
  Moreover, since each $S_j^*$ leaves $\Lambda(\ell)$
  invariant, the subspace $\mathbb F(\ell)^2$ of $\mathbb F^2$ spanned by
  $\Lambda(\ell)$ is invariant for $S^*$.  Let
  $V(\ell)$ denote the inclusion of  $\mathbb F(\ell)^2$
  into $\mathbb F^2$  and
  let $S(\ell)$ denote the operator $V(\ell)^* S  V(\ell)$.
  Thus, $S(\ell)=((S(\ell))_1,\dots,(S(\ell))_d)$ where
  $(S(\ell))_j = V(\ell)^* S_j V(\ell)$.  Observe,
  with $P$ denoting both the projection of $\mathbb F^2$
  and $\mathbb F(\ell)^2$ onto the orthogonal complement of the
  span of $\emptyset$ in $\mathbb F^2$ and $\mathbb F(\ell)^2$ respectively,
  equation \eqref{eq:co-isometry} yields
\begin{equation*}
 \begin{split}
   P= & V(\ell)^* P V(\ell) \\
     =& V(\ell)^* \left(\displaystyle \sum_{j=1}^d  S_j S_j^*\right) V(\ell) \\
     =& \displaystyle \sum_{j=1}^d (S(\ell))_j (S(\ell))_j^*.
 \end{split}
\end{equation*}
   It follows, for $t<\gamma$, that $tS(\ell) \in \cN_\gamma(n)$,
   where $n=\sum_{j=0}^\ell d^j$ is the dimension of   $\mathbb F(\ell)^2$.

\subsection{The Algebra $\cAK$}
\label{subsec:matrixproduct}
   In addition to the obvious pointwise addition and multiplication
   by scalars,  there is a natural
   multiplication on $\cAK$ extending multiplication of
   polynomials which then turns $\cAK$ into an algebra
   over $\mathbb C$.  Since it will be necessary to consider,
   in the sequel, matrices with entries from $\cAK$ and their
   products, we define them here.

   Let
\[
  M_{p,q}(\cAKn)
   = \left\{f = \displaystyle\sum_{w \in F_d}f_w w : \, f_w\in M_{p,q},
  \mbox{ and for each } X \in \cK, \,
  f(X) \, \text{ converges }\right\}.
\]
 where

 \[
 f(X) = \displaystyle\sum_{j = 0}^{\infty} \sum_{|w|=j} f_w \otimes X^w.
\]

   For $f\in M_{p,q}(\cAKn)$, let
 \begin{equation}
  \label{def:normpq}
  \|f \|=\|f\|_{p,q} =\sup\{\|f(X)\|: \, X \in \cK\}.
 \end{equation}

Define
\begin{equation*}
 M_{p,q}(\cAK) = \left\{f = \displaystyle\sum_{w \in F_d}f_ww : f \in
 M_{p,q}(\cAKn), \|f\| < \infty \right\}.
\end{equation*}

   The following Lemma plays an important role in the analysis
   to follow generally, and in proving that $\cAK$
   is an algebra, in particular.

\begin{lemma}
 \label{lem:A-bounded}
  Suppose that $f = \displaystyle \sum_{w \in \cFd}f_w w \in M_{p,q}(\cAK)$
  and $X \in \mathcal{K}(n)$.
  Let
 \[
   A_j =\sum_{|w|=j} f_w\otimes X^w.
 \]
   If $0 < r < \sup\{0<s : sX\in\mathcal K(n)\}$, then
 \[
   r^j \|A_j \| \le \|f\|.
 \]
  In particular, there is a $\rho<1$ such that $\|A_j\|\le \rho^j \|f\|$.
\end{lemma}

\begin{proof}
  Because $\mathcal K(n)$ is open, convex, and circled, the
  function $F(z)=f(zX)$ is defined on a neighborhood of
  the closed unit disk $\{|z|\le 1\}$. Thus the series,
\[
  F(z)= \displaystyle \sum_{j=0}^{\infty} A_j z^j
\]
  has radius of convergence exceeding one.  Thus,
  for each $j\in\mathbb N$,
\[
  A_j=\frac{1}{2\pi} \int_0^{2\pi} F(e^{it}) e^{-ijt}\, dt.
\]
  It follows that
\[
  \|A_j\| \le \frac{1}{2\pi} \int_0^{2\pi} \|F(e^{it})\|\, dt.
\]
  Since $\|F(e^{it})\|=\|f(e^{it}X)\|$ and $e^{it}X \in \mathcal{K}(n)$,
  it follows that $\|F(e^{it})\|\le \|f\|$ and the lemma follows.
\end{proof}

Given
$f = \sum_{w \in \cFd} f_{w} w \in M_{p, q} (\cAK)$
and $g =\sum_{w \in \cFd} g_{w} w \in M_{q, r} (\cAK)$,
define the product $fg$ of $f$ and $g$ as
 the convolution product; i.e.,
 \[
    fg = \displaystyle \sum_{w \in \cFd}
     \left(\sum_{uv = w} f_ug_v \right)  w.
\]
 This convolution product corresponds to pointwise product,
 extends the natural product
 of non-commutative polynomials (formal power series with only
 finitely many non-zero coefficients), and makes
 $\cAK$ an algebra.

\begin{lemma}
  \label{lem:product}
    If $f\in M_{p,q}(\cAK)$ and $g\in M_{q,r}(\cAK)$
    and $X \in \cK$, then
 \begin{itemize}
   \item[(i)] $fg(X)$ converges;
   \item[(ii)] $fg(X)=f(X)g(X)$;
   \item[(iii)] $fg$ is in $M_{p,r}(\cAK);$ and
   \item[(iv)]  $\|fg\|\le \|f\|\|g\|$.
 \end{itemize}
\end{lemma}

\begin{corollary}
\label{cor:algebra}
  $\cAK$ is an algebra.
\end{corollary}

\begin{proof}[Proof of Lemma {\ref{lem:product}}]
   Fix $X \in \cK$. As in the proof of
  Lemma {\ref{lem:A-bounded}}, let
\[
 \begin{split}
  A_j=&\sum_{|w|=j} f_w\otimes X^w, \\
  B_j=&\sum_{|w|=j} g_w \otimes X^w \\
  C_j=&\sum_{|w|=j} (\sum_{uv=w} f_ug_v)\otimes X^w.
 \end{split}
\]
  Observe that $C_j =  \sum_{k=0}^j A_k  B_{j-k}$.

  Let $F(z)=f(zX)$ and $G(z)=g(zX)$, both of
  which are defined in a neighborhood of $\{|z|\le 1\}$.
  From Lemma {\ref{lem:A-bounded}}, there is an $\rho<1$
  such that $\|A_m\|\le \rho^m \|f\|$
  and $\|B_k\|\le \rho^k \|g\|$.  Hence
 \[
   \|C_j\|\le (j+1) \|f\|\, \|g\| \rho^j.
 \]
 It follows that, for $|z|<\frac{1}{\rho}$, the series
 \[
   \displaystyle \sum_{j=0}^{\infty} (\sum_{k=0}^j A_k B_{j-k}) z^j
 \]
   converges absolutely.  In particular $fg(X)=
   \sum_{j=0}^{\infty} C_j$
   converges in norm.

  For $|z|<\frac{1}{\rho}$, one verifies that
\begin{equation*}
 \begin{split}
  F(z)G(z)= & \displaystyle \sum_{j=0}^{\infty} \sum_{k=0}^j A_k B_{j-k} z^j \\
          = & fg(zX).
 \end{split}
\end{equation*}
  Choosing $z=1$ gives $f(X)g(X)=fg(X)$.

  Since, for each $X \in \cK$,
  $fg(X)=f(X)g(X)$ it follows that $\|fg(X)\|\le \|f\|\, \|g\|$.
  Thus $\|fg\|\le \|f\|\, \|g\|$ and $fg\in M_{p,r}(\cAK)$.
\end{proof}

\subsection{The Abstract Unital Operator Algebra $\cAK$}
In this section, for the convenience of the reader, the definition of an
abstract operator algebra is reviewed. Following that, it is shown
that $\cAK$ with the norms $\|\cdot\|_{p,q}$ on $M_{p,q}(\cAK)$
given in equation \eqref{def:normpq} is an abstract operator algebra.

\subsubsection{Abstract Operator Algebra}
 Let $V$ be a complex vector space and $M_{p,q}(V)$ denote
 the set of all $p \times q$ matrices with entries from $V$.
 $V$ is said to be a {\it matrix normed space}
 \index{matrix normed space}  provided that there exist norms $\|\cdot\|_{p,q}$
 on $M_{p , q} (V)$ that satisfy
 \begin{equation*}
 \|A \cdot X \cdot B\|_{\ell,r} \leq \|A\| \|X\|_{p,q}  \|B\|
 \end{equation*}
 for all $A \in M_{\ell,p}, X \in M_{p,q}(V), B \in M_{q,r}$.

A matrix normed space $V$ is said to be an {\it abstract operator space}
\index{abstract operator space} if
\begin{equation*}
\|X \oplus Y\|_{p+ \ell,q+r} = \max \{\|X\|_{p,q} , \|Y\|_{\ell,r}\}
\end{equation*}
where $X \in M_{p,q}(V)$ and $Y \in M_{\ell,r}(V)$ and
$X \oplus Y = \begin{pmatrix}X & 0 \\ 0 & Y\end{pmatrix}$.

$V$ is an {\it abstract unital operator algebra} \index{abstract operator algebra}
if $V$ is a unital algebra, an abstract operator space and if the product on $V$
is completely contractive i.e. $\|XY\|_p \leq 1$ whenever $\|X\|_p \leq 1$
and $\|Y\|_p \leq 1$ for all $X, Y \in M_p(V)$ and for all $p$.

\subsubsection{The Abstract Unital Operator Algebra $\cAK$}

Consider $\cAK$ with $\| \cdot\|_{p,q}$ being
the norm on $M_{p,q}(\cAK)$ defined in subsection
{\ref{subsec:matrixproduct}} (see equation \eqref{def:normpq}).

\begin{theorem}
$\cAK$ with the family of norms $\| \cdot \|_{p,q}$, is an abstract
unital operator algebra.
\end{theorem}
\begin{proof}
Let $A \in M_{\ell,p}$, $F \in M_{p,q}(\mathcal{A^\infty})$, $B \in M_{q,r}$.
Interpret $A$ and $B$ as $A\emptyset \in M_{\ell,p}(\cAK)$
and $B \emptyset \in M_{q,r}(\cAK)$ respectively. For
notation ease we will drop the subscripts that go with the norms.
It follows from
Lemma {\ref{lem:product}}(ii) that for all $X \in \mathcal{K}(n)$,
\begin{equation*}
\|AFB(X)\| = \|A(X)F(X)B(X)\| \leq \|A \otimes I_n\|\, \|F(X)\| \, \|B \otimes I_n\|
 \leq \|A\| \, \|F\| \, \|B\|.
\end{equation*}
Thus,
\begin{equation}
\| A F B\|  \leq \|A\|\, \|F\|\, \|B\|.
\end{equation}
Let $F \in M_{\ell, r}(\cAK), \, G \in M_{p,q}(\cAK), \, X \in \mathcal{K}(n)$.
Observe that
\begin{equation*}
\|F \oplus G \, (X) \| =\left\Vert \begin{pmatrix}F(X) & 0 \\ 0 &
G(X)\end{pmatrix}
\right\Vert \leq \max \{\|F(X)\|,\|G(X)\|\} \leq  \max\{\|F\|, \|G\|\}.
\end{equation*} Thus,
\begin{equation}
\label{eq:lessthan}
\|F \oplus G\| \leq  \max\{\|F\|, \|G\|\}
\end{equation}
Let $\epsilon > 0$ be given. Without loss of generality assume that
$\|F\| \geq \|G\|$.
 Choose $m \in \N$ and $R \in \mathcal{K}(m)$ such that $\|F(R)\|
 > \|F\| - \epsilon$.
 Therefore
\begin{equation}
\label{eq:greaterthan}
\| F \oplus G \| \geq \left\Vert \begin{pmatrix} F(R) & 0 \\ 0 &
G(R)\end{pmatrix} \right\Vert
\geq \|F(R)\| >  \|F\| - \epsilon.
\end{equation}
Letting $\epsilon \rightarrow 0$ in the inequality ({\ref{eq:greaterthan}}) and from
the inequality ({\ref{eq:lessthan}})
it follows that,
\begin{equation}
\|F \oplus G\| = \max\{\|F\| , \|G\|\}.
\end{equation}
Lastly, complete contractivity of multiplication in $M_p(\cAK)$
follows directly from
Lemma {\ref{lem:product}} (iv). Thus $\cAK$ is an abstract
operator algebra. $\emptyset \in \cAK$ is the multiplicative unit.
\end{proof}

\section{Weak Compactness and  $\cAK$}
 \label{sec:weak}
 In this section it is shown that every bounded
 sequence in $\cAK$ has a pointwise
 convergent subsequence.
 Indeed, $\cAK$ has weak compactness properties
 with respect to bounded pointwise convergence
 mirroring those for $H^\infty$, the usual
 space of bounded analytic functions on
  the unit disk  $\mathbb D$.

 The results easily extend to formal power
 series with matrix coefficients and it
 is at this level of generality that they
 are needed in the sequel.

\begin{proposition}
 \label{prop:bounded-pointwise}
  Suppose that $f_m=  \sum_{w\in\cFd} (f_m)_w w$ is a
  $M_{p,q}(\cAK)$ sequence.  If, for each $X\in\cK$
  the sequence $f_m(X)$ converges or
  if for each $w \in \cFd$ the sequence $(f_m)_w$ converges,
  and if $(f_m)$
  is a bounded sequence (so there is a constant $c$ such
  that $\|f_m\|\le c$ for all $m$), then there is
  an $f\in M_{p,q}(\cAK)$ such that $f_m(X)$ converges to
  $f(X)$ for each $X \in\cK$ and moreover $\|f\|\le c$.
\end{proposition}

\begin{proof}
  If $f_m$ converges pointwise, then, by considering
  $f_m(S(\ell))$ where $S(\ell)$ is defined in
  Subsection {\ref {subsec:genesis}}, it follows that
  $(f_m)_w$ converges to some $f_w$ for
  each word $w$.  Hence, to prove the Proposition
  it suffices to prove, if $(f_m)_w$ converges
  to $f_w$ for each $w$ and $\|f_m\|\le c$
  for each $m$, then for each $X\in\cK$,
  the series
 \[
   f(X)=\sum_{j=0}^\infty \sum_{|w|=j} f_w \otimes X^w
 \]
  converges and $(f_m(X))$ converges to $f(X)$.

  For any $X\in \cK$,
\begin{equation*}
\displaystyle\sum_{|w|=j} (f_{m})_w \otimes X^w
   \rightarrow \displaystyle\sum_{|w|=j} f_w \otimes X^w.
\end{equation*}
From Lemma {\ref{lem:A-bounded}}, there is a $\rho < 1$
 such that for all $j \in \N$,
\begin{equation*}
\| \displaystyle\sum_{|w|=j} (f_{m})_w \otimes X^w\|
    \leq \rho^j c,
\end{equation*}
 an estimate from which the conclusions of the
 proposition are easily seen to follow.
\end{proof}

\begin{lemma}
\label{lem:subsequence}
  If $f_m =\sum_{w \in \cFd}(f_m)_w  w \in M_{p,q}(\cAK)$
  satisfies $\|f_m\| \leq c$ for all $m \in \N$ then,
\begin{itemize}
   \item[(i)]  $\|(f_m)_w\| \leq \frac{c}{\gamma^{|w|}}$
 for all $w \in \cFd$ and for all $m \in \N$;
   \item[(ii)]  There exists a subsequence $\{f_{m_k}\}$ of $\{f_m\}$
    and $f_w \in M_{p,q}$ such that
   ${(f_{m_k})}_w \rightarrow f_w$ for all $w$;
   \item[(iii)] Let
    $f=\displaystyle \sum_{w \in \cFd} f_w w$.  For
    each $X\in\cK$ the sequence $(f_{m_k}(X))$ converges
    to $f(X)$ and moreover $\|f(X)\|\le c$.
 \end{itemize}
\end{lemma}

\begin{proof}
To prove item (i), Recall $\gamma$ from the definition
of $\mathcal{K}$. Let $t < \gamma$ and $x \in \C^q$ be a unit vector.
For $j = 0, 1, 2, ..., \ell$, the hypothesis
$\|f_m\| \leq c$ together with the conclusion of
Lemma {\ref{lem:A-bounded}} for $X = tS({\ell})$
imply that

\begin{equation*}
\| \displaystyle  t^{j} \sum_{|w| = j} (f_m)_w  \otimes  S({\ell})^w \| \leq c.
\end{equation*}

Hence
\begin{align*}
c^2 &\geq  \| \displaystyle  t^{j} \sum_{|w| = j} (f_m)_w x
\otimes  S({\ell})^w \emptyset \|^2\\
    &= t^{2j}  \displaystyle\sum_{|w| = j} \|(f_m)_w x\|^2\\
    &\geq t^{2j} \|(f_m)_w x\|^2.
\end{align*}
Since $x$ and $\ell$ are arbitrary, letting $t \uparrow \gamma$
it follows that
 $\|(f_m)_w\| \leq \frac{c}{\gamma^{|w|}}$ for all $m \in \N$.

The proof of item (ii) uses a standard diagonal  argument.
Let $\{w_1, w_2, ...\}$ be an enumeration
of words in $\cFd$ which respects length (i.e.,
if $v\le w$, the $|v|\le |w|$).
Since $\|(f_m)_{w_1} \| \leq \frac{c}{\gamma^{|w_1|}}$, there
exists a subsequence say, $\{f_{1,m}\}$ of $\{f_m\}$ such
 that $(f_{1,m})_{w_1} \rightarrow f_{w_1}$.
Since $\|(f_{1,m})_{w_2}\| \leq \frac{c}{\gamma^{|w_2|}}$,
 there exists a subsequence say,
$\{f_{2,m}\}$ of $\{f_{1,m}\}$ and thereby of $\{f_m\}$,
such that $(f_{2,m})_{w_2} \rightarrow f_{w_2}$.
Continue this procedure to obtain a subsequence $\{f_{k,m}\}$
of $\{f_{k-1,m}\}$ and thereby of $\{f_m\}$
such that for all $k \in \N$,
\begin{equation*}
(f_{k,m})_{w_k} \rightarrow f_{w_k}.
\end{equation*}

 Now consider the diagonal sequence $\{f_{m,m}\}$.
 It follows that $\{f_{m,m}\}$ is a subsequence of $\{f_m\}$
 and satisfies $(f_{m,m})_w \rightarrow f_w$ for all $w \in \cFd$.

 In view of what has already been proved,  an
 application of Proposition {\ref{prop:bounded-pointwise}}
 proves item (iii).
\end{proof}

\section{The operator algebra  $\cAK / \cIK$}
 \label{sec:quotient}
 In this section we consider the ideal $\cIK$  of the
 algebra $\cAK$ determined by a finite inital
 segment $\Lambda$. It is shown that the quotient algebra
 $\cAK/\cIK$ is infact an abstract unital operator algebra.
 It is also established that norms of classes in the
 quotient algebra are attained.

\subsection{The Abstract Unital Operator Algebra $\cAK / \cIK$}
For the initial segment $\Lambda \subset \cFd$, let
\begin{equation*}
\cIK = \left\{ f = \displaystyle\sum_{w \not\in \Lambda} f_w w \quad : \quad
\|f\| < \infty \right\} \subset \cAK
\end{equation*}
Observe that $\cIK$ is a closed two-sided ideal in the operator algebra $\cAK$.
 The usual identification of
 $M_{p, q}(\cAK / \cIK)$
with $M_{p, q}(\cAK) / M_{p, q}(\cIK)$
 yields the well known fact that the
 quotient of an abstract operator algebra
  by a closed two sided ideal is again
 an abstract operator algebra (see Exercises 13.3 \& 16.3 in \cite{P}).
  We formally record this fact.

\begin{theorem}
 \label{thm:quotient-op-alg}
 $\cAK / \cIK$ is an abstract unital operator algebra.
\end{theorem}

\subsection{Attainment of Norms of Classes in $M_q(\cAK) / M_q(\cIK)$}
Let $p \in M_q(\cAK)$. In this section it is shown that there exists
$f \in M_q(\cIK)$ such that
\begin{equation*}
\|p+f\| = \|p + M_q(\cIK)\| = \inf\{\|p + g\| : g \in M_q(\cIK)\}.
\end{equation*}
Let $\{f_m\}$ be a sequence in $M_q(\cIK)$ such that
\begin{equation*}
\label{eq:approximations}
\|p + M_q(\cIK)\| \le \|p + f_m\| \le \|p + M_q(\cIK)\| + \frac{1}{m}
\end{equation*}
It follows that the sequence $\{f_m\}$ is bounded and that
$\|p + f_m\| \rightarrow \|p+M_q(\cIK)\|$. An application of Lemma {\ref{lem:subsequence}}
yields a subsequence $\{f_{m_k}\}$ of $\{f_m\}$ and $f \in M_q(\cIK)$ such that
\begin{equation*}
(p+f_{m_k})(X) \rightarrow (p+f)(X)
\end{equation*}
for all $X \in \cK$.

\begin{proposition}
\label{thm:normattained}
If $p, \{{f_m}_k\}, f$ are as above, then $\|p \, + \, f\| = \|p \, + \, M_q(\cIK)\|$.
\end{proposition}

\begin{proof}
Let $\epsilon > 0$ be given. Choose $R \in \cK$ such that
\begin{equation}
\label{eq:lt1}
\|p + f\| < \|(p +f)(R)\| + \frac{\epsilon}{4}.
\end{equation}

Since $\|(p + f_{m_k})(R)\| \rightarrow \|(p + f)(R)\|$, there exists $K_1 \in \N$
such that,
\begin{equation}
\label{eq:lt2}
\|(p + f)(R)\| < \|(p + f_{m_k})(R)\| + \frac{\epsilon}{4}
\end{equation}
for all $k \ge K_1$.
Combining the inequalities from equations {\ref{eq:lt1}} and
{\ref{eq:lt2}}, implies that, for all $k\geq K_1$,
\begin{equation}
\label{eq:lt3}
\|p + f\| < \|p + f_{m_k}\| + \frac{\epsilon}{2}.
\end{equation}

Since $\|p + f_{m_k}\| \rightarrow \|p + M_q(\cIK)\|$, there exists
a Natural number $K_2$
 such that for all $k \geq K_2$,
\begin{equation}
\label{eq:lt4}
\|p + f_{m_k}\| < \|p + M_q(\cIK)\| + \frac{\epsilon}{2}.
\end{equation}

Setting $k$ = $\max \{K_1,K_2\}$ in equations ({\ref{eq:lt3}}) and ({\ref{eq:lt4}}),
and letting $\epsilon \rightarrow 0$ yields

\begin{equation*}
\|p + f\| \leq \|p + M_q(\cIK)\|.
\end{equation*}

On the other hand, since $f \in M_q(\cIK)$,

\begin{equation*}
\|p + f\| \geq \|p + M_q(\cIK)\|.
\end{equation*}

\end{proof}

\section{Representations of the operator algebra $\cAK$}
\label{sec:reps}

In this section it is shown that completely
contractive representations of the algbera $\cAK$, when
compressed to finite-dimensional subspaces end up
in the boundary of the matrix-convex set $\cK$.

\subsection{Completely Contractive/Isometric Representation}

Let $\cV$ and $\cW$ be abstract operator spaces and $\phi: \cV \rightarrow \cW$ be
a linear map.
Define $\phi_q: M_q \otimes \cV \rightarrow M_q \otimes \cW$ by
$\phi_q = I_q \otimes \phi$, where $I_q$ is the $q \times q$ identity
matrix.

The map $\phi$ is said to be {\it completely contractive (isometric)}
\index{completely contractive (isometric)} if $\phi_q$ is a contraction (isometry)
 for each $q \in \N$.

A {\it completely contractive (isometric)} representation of an algebra $\cA$ is a
completely contractive (isometric) algebra homomorphism $\pi: \cA \rightarrow B(\cM)$
for some Hilbert space $\cM$.\\

The following Theorem due to Blecher, Ruan and Sinclair guarantees the
existence of a completely isometric Hilbert space
representation for an abstract unital
operator algebra.

\begin{theorem}[BRS]
\label{thm:BRS}
Let $\cA$ be an abstract unital operator algebra. There exists a Hilbert space $\cM$
and a unital completely isometric algebra homomorphism $\pi: \cA \rightarrow B(\cM)$,
i.e. a unital operator algebra isomorphism onto $\pi(\cA)$.
\end{theorem}

\subsection{Completely Contractive Representations of $\cAK$}
 Given a completely contractive unital
 representation  $\pi:\cAK\to\mathcal B(\cM),$
 let $T_j=\pi(g_j)$ and let $T=(T_1,\dots,T_d)$.
 For notation convenience,
 we will write $\pi_T$ for $\pi$.
 Further, we will also use  $\pi_T$ to denote the map
 $I_q \otimes \pi: M_q(\cAK) \to M_q\otimes B(\cM)$.

 A main result of this
 section says, for a completely contractive representation $\pi_T$
  of $\cAK$,
  for any $n \in \N$ and finite dimensional subspace $\cH$ of $\cM$
  of dimension $n$
  and $0\le t<1$ the tuple
\[
   tZ=tV^*T V=(tV^*T_1V,\dots,tV^*T_dV)
\]
  is in $\mathcal K(n)$.
 The proof begins with a couple of lemmas.
 Given $f\in M_q(\cAK)$ and $0\le r<1$, let $f_r$ be
 defined as follows.
If
\begin{equation}
 \label{eq:formf}
   f=\sum_{j=0}^{\infty} \sum_{|w|=j} f_w w =\sum_{j=0}^{\infty} f_j.
\end{equation}
 then
\[
  f_r=\sum_{j=0}^{\infty} r^j \sum_{|w|=j} f_w w = \sum_{j=0}^{\infty} r^j f_j.
\]

\begin{lemma}
\label{lem:convergence}
  If $\pi_T$ is a completely contractive representation
  of $\cAK$ and $f\in M_q(\cAK)$, then $f_r(T)$ converges
  in operator norm. Moreover $\pi_{T}(f_r) = f_r(T)$ and
  $\|f_r(T)\| \le \|f_r\|\le \|f\|$.

  If in addition $\pi_T$ is completely isometric, then
  $\lim_{r\to 1^-} \|f_r(T)\|=\|f\|$.
\end{lemma}

\begin{proof}
 Write $f$ as in equation \eqref{eq:formf}.
 Lemma {\ref{lem:A-bounded}} implies that $\|f_j\|\le\|f\|$.
 Because $\pi_{T}$ is completely contractive $\|f_j(T)\|\le \|f\|$.
 It follows that $f_r(T)$ converges in norm.
 Since also the partial sums of $f_r$ converge
 (to $f_r$) in the norm of $M_q(\cAK)$, it follows that
 $\pi_{T}(f_r) = f_r(T)$ and so $\|f_r(T)\|\le \|f_r\|$.

 The inequality $\|f_r\|\le \|f\|$ is straightforward
 because $r\cK\subset \cK$.

 Now suppose that $\pi_T$ is completely isometric.
 In this case $\|f_r(T)\|=\|f_r\|$.
 On the other hand $  \lim_{r\to 1^-} \|f_r\|=\|f\|$.
\end{proof}

\begin{lemma}
 \label{lem:transform}
  Given  $A_1,\dots,A_d$ are $k\times k$ matrices.
  let
 \[
   L=\sum_{j=1}^d A_j g_j.
 \]

  Suppose
 \[
   2 - L(X)-L(X)^* \succ 0
 \]
 for all $X \in \mathcal{K}(\ell)$ and for all $\ell \in \N$.
  Let $\Phi_L$ denote the formal power series,
 \[
   \Phi_L= L(2-L)^{-1} = \displaystyle\sum_{j=0}^\infty
   \frac{L^{j+1}}{2^{j+1}}.
 \]

  (a) If $X \in \mathcal{K}(\ell)$, then
  $\Phi_L(X)$ converges in norm; i.e., the
  series
\[
    \displaystyle\sum_{j=0}^\infty
   \frac{L(X)^{j+1}}{2^{j+1}}
\] converges.

  (b) $\|\Phi_L(X)\| < 1$ and hence $\Phi_L$
  is in $M_k(\cAK)$ and has norm at most one.

   (c) If $\pi_T$ is a completely contractive representation
   of $\cAK$,
  then $2-(L(T)+L(T)^*)\succeq 0$.
\end{lemma}

\begin{proof}
  To prove part (a) of the lemma, let $X\in\mathcal K(\ell)$ be given.
  Because $\mathcal K(\ell)$ is circled,
  it follows that $e^{i\theta}X\in \mathcal K(\ell)$
  for each $\theta$. Hence,
\begin{equation}
 \label{eq:num1}
  2-e^{i\theta} L(X)- e^{-i\theta} L(X)^* \succ 0
\end{equation}
 for each $\theta$.
  For notation ease, let $Y=L(X)$. Thus
  $Y$ is a $k\ell\times k\ell$ matrix
  and equation \eqref{eq:num1} implies that
  the spectrum of $Y$ lies strictly within the disc;
  i.e., each eigenvalue of $Y$ has absolute value
  less than one.  Thus,
 \[
    \frac{1}{2} \displaystyle\sum_{j=0}^{\infty}
    \left(\frac{Y}{2}\right)^j = (2-Y)^{-1}
 \]
  converges in norm.  It follows that
 \[
   \Phi_L(X)=Y(2-Y)^{-1}
     = \displaystyle \sum_{j=0}^\infty
   \frac{Y^{j+1}}{2^{j+1}}
 \]
  converges.

  To prove (b) observe that
  $\|Y(2-Y)^{-1}\|< 1$ if and only if
 \[
   (2-Y)^* (2-Y) \succ Y^* Y
 \]
  which is equivalent to $2-(Y+Y^*)\succ 0$. Thus $\|\Phi_L(X)\| < 1$
  which implies that $\Phi_L \in M_k(\cAK)$ with $\|\Phi_L\|\le 1$.
 This completes the proof of (b).

 To prove part (c), observe,
 Since $\pi_T$ is completely contractive and
 $\Phi_L\in M_k(\cAK)$ with norm at most one,
 an application of Lemma {\ref{lem:convergence}} yields,
 $\| \Phi_L(rT) \|\le 1$. Arguing
 as in the proof of part (b), it follows that
 $2-(L(rT)+L(rT)^*) \succeq 0$. This inequality
  holds for all $0\le r <1$ and thus the conclusion
 of part (c) follows.
\end{proof}

\begin{proposition}
\label{prop:cutdown}
   If $T=(T_1,\dots,T_d),$
   and $T_j\in\mathcal B(\cM)$ for
  some Hilbert space $\cM,$
   and $\pi(g_j)=T_j$ determines
  a completely contractive representation
  of $\cAK$, then, for each positive
  integer  $n$ and finite
  dimensional subspace $\cH$ of $\cM$
  of dimension $n$ and each $0\le t<1$
  the tuple
\[
   tZ=tV^*T V=(tV^*T_1V,\dots,tV^*T_dV)
\]
 is in $\mathcal K(n)$, where
  $V:\cH\to\cM$ is the inclusion map.
\end{proposition}

\begin{proof}
  Let $n$ and $\cH$ be given and define
  $Z$ as in the statement of the proposition.
  Suppose that $L$ is as in the statement of
  Lemma {\ref{lem:transform}}.
  From part (c) of the previous lemma,
  it follows that $2-(L(T)+L(T)^*)\succeq 0$.
  Applying $I_k \otimes V^*$ on the left and $I_k \otimes V$ on the right
  of this inequality gives,
 \[
  2-(L(Z)+L(Z)^*) =(I_k \otimes V^*)(2-(L(T)+L(T)^*)(I_k \otimes V) \succeq 0.
 \]
 An application of Theorem 5.4 from \cite{EW} implies that
 $Z \in \overline{\cK(n)}$. Hence
 $tZ \in \cK(n)$ for all $0 \le t <1$.
\end{proof}

\begin{lemma}
\label{lem:normbound}
Let $\Lambda \subset \cFd$ be a finite initial segment,
$f \in M_q(\cAK)$ be as in equation ({\ref{eq:formf}}) and
suppose that $\pi_T$ is a completely contractive
representation of $\cAK$ into $B(\cM)$ and
$T$ is $\Lambda-\mbox{nilpotent}$.
Then $\|f_r(T)\| \leq \sup\{\|f(X)\|: X \in \cK, X \mbox{ is } \Lambda-\mbox{nilpotent}\}$
for all $0 \le r <1$. Moreover if $f_w = 0$ for all $w \not \in \Lambda$, then
$\|f(T)\| \leq \sup\{\|f(X)\|: X \in \cK, X \mbox{ is } \Lambda-\mbox{nilpotent}\}$.
\end{lemma}

\begin{proof}
Since $\Lambda$ is finite, $f_r(T) =
\sum_{w \in \Lambda} f_w \otimes (rT)^w$.
Let $\{e_j\}_{j=1}^q$ be the standard basis of $\C^q$ and
 $y =  \sum_{j = 1}^q e_j \otimes h_j \in \C^q \otimes \cM$
 be a unit vector such that
\[
\|f_r(T)\| < \|f_r(T)y\| + \epsilon.
\]

Let $\cH$ denote the finite-dimensional subspace of $\cM$ spanned by the vectors
$\{T^w(h_j) : w \in \Lambda, 1 \le j \le q\}$ and $V: \cH \rightarrow \cM$
be the inclusion map. Then $Z = V^*TV$ is $\Lambda$-nilpotent and
\[
Z^w = \begin{cases} V^*T^wV & \text{if }w \in \Lambda\\
                                       0 & \text{otherwise. }
     \end{cases}
\]
Proposition {\ref{prop:cutdown}} implies that $rZ \in \cK$. Thus,
\begin{align*}
\|f_r(T)\| &< \|  \left(\sum_{w \in \Lambda} f_w \otimes
            r^{|w|} T^w\right)y \| + \epsilon\\
           &= \|f_r(Z)y\| + \epsilon\\
           &\le \|f_r(Z)\| + \epsilon\\
           &\le \sup\{\|f(X)\|: X \in \cK, X \mbox{ is }
           \Lambda-\mbox{nilpotent}\} + \epsilon.
\end{align*}
Letting $\epsilon \rightarrow 0$ yields the desired inequality.
If $f_w = 0$ for all $w \not \in \Lambda$, then $f =
\sum_{w \in \Lambda} f_w w$
is a non-commutative polynomial in which case we have
$ \lim_{r \rightarrow 1^-} \|f_r(T)\| = \|f(T)\|$
and this completes the proof.
\end{proof}

\section{The Caratheodory-Fejer Interpolation Problem (CFP)}
 \label{sec:opversion}
 The proof of the generalization of Theorem \ref{thm:main}
 allowing for operator coefficients is proved in
 this section.

 The strategy is to first prove
  the result for matrix coefficients.
 This is done in Subsection \ref{subsec:matrix-version}
 below.  Passing from matrix to operator
 coefficients is then accomplished using
 well-known facts about the Weak Operator Topology (WOT)
  and the Strong Operator Topology (SOT) on the space
 of bounded operators on a separable Hilbert space.
 The details are in Subsection \ref{subsec:operator-version}.

\subsection{Matrix version of the CFP}
 \label{subsec:matrix-version}
Fix $\Lambda \subset \cFd,$ a finite initial segment,
and polynomial
$p =  \sum_{w \in \Lambda} p_w w \in M_q(\cAK)$.

Theorem \ref{thm:main} is easily seen to follow
from the following proposition.

\begin{proposition}
\label{thm:matrixversion}
There exists $f \in M_q(\cAK)$ such that $\|p + f\| =
\|p + M_q(\cIK)\|  = \sup\{\|p(X)\|: X \in \cK, X \mbox{ is }
\Lambda-\mbox{nilpotent}\}$.
\end{proposition}

\begin{proof}
From Theorems \ref{thm:quotient-op-alg} and
{\ref{thm:BRS}} it follows that there exists
a Hilbert space $\cM$ and a completely isometric homomorphism
$\pi: \cAK / \cIK \, \rightarrow B(\cM)$. As before, identify
$M_q(\cAK / \cIK)$ with $M_q(\cAK) / M_q(\cIK)$. Let $\pi_q$ denote
the map $I_q \otimes \pi: M_q(\cAK) / M_q(\cIK) \rightarrow M_q \otimes B(\cM)$.
Let $R$ be the d-tuple $(R_1, R_2, ..., R_d)$, where
$R_j = \pi(g_j + \cIK) \in B(\cM)$, for $1 \le j \le d$.
Observe that $R$ is $\Lambda$-nilpotent.
Let $\eta: \cAK \rightarrow \cAK/\cIK$ be the quotient map. The
 composition map $\pi \circ \eta : \cAK \rightarrow B(\cM)$
is a completely contractive representation of $\cAK$. Also since
$\pi(\eta(g_j)) = R_j$, consistent with the notation introduced
earlier, we will use $\pi_R$ to denote the map
$\pi \circ \eta$.

 It follows from Theorem {\ref{thm:normattained}}
 that there exists $f \in M_q(\cIK)$
 such that
 \begin{equation}
 \label{eq:normattained}
 \|p + f\| = \|p + M_q(\cIK)\|.
 \end{equation}
 The fact that $\pi$ is completely isometric implies that
 \begin{equation}
 \label{eq:completeisom}
 \|p + M_q(\cIK)\| = \|\pi_q(p + M_q(\cIK)\| = \|p(R)\|
 \end{equation}
 Since $\pi_R$ is a completely contractive representation of $\cAK$, Lemma
 {\ref{lem:normbound}} implies that
 \begin{equation}
 \label{eq:bound}
 \|p(R)\| \leq \sup\{\|p(X)\|: X \in \cK, X \mbox{ is } \Lambda- \mbox{nilpotent}\}.
 \end{equation}
Combining the equations ({\ref{eq:normattained}}), ({\ref{eq:completeisom}})
and ({\ref{eq:bound}}), it follows that
\[
\|p + f\| \le  \sup\{\|p(X)\|: X \in \cK, X \mbox{ is } \Lambda- \mbox{nilpotent}\}.
\]
But the definition of $\|p +f\|$ implies that
\[
\|p + f\| \ge \sup\{\|p(X)\|: X \in \cK, X \mbox{ is } \Lambda- \mbox{nilpotent}\}
\]
and this completes the proof.
\end{proof}

\subsection{Operator Version of the CFP}
 \label{subsec:operator-version}
 As before,
 let $\Lambda \subset \cFd$ be a finite initial segment.
 Departing from the previous subsection, let  $\cU$
be a separable Hilbert space and let the polynomial
$p =  \sum_{w \in \Lambda} p_w w$,
where $\{p_w\}_{w \in \Lambda} \subset B(\cU)$
be given.

\begin{theorem}
\label{thm:operatorversion}
 There exists a formal power series
 $\tilde{x} =  \sum_{w \in \cFd} \tilde{x}_w w$ such that
 $\tilde{x}_w = p_w$ for all $w \in \Lambda$ and
 $\|\tilde{x}\| = \sup\{\|p(X)\|: X \in \cK, X \mbox{ is }
 \Lambda-\mbox{nilpotent}\}$.
\end{theorem}

\begin{proof}
Let $\{u_1, u_2, ...\}$ denote an orthonormal basis for the separable Hilbert
space $\cU$ and $\cU_m$ be the subspace of $\cU$ spanned by the vectors
$\{ u_j \}_{j=1}^m$. For notation ease, let
$C = \sup\{\|p(X)\|: X \in \cK, X \mbox{ is } \Lambda-\mbox{nilpotent}\}$.

For $w \in \Lambda$, define $M_m \ni (p_m)_w = V_m^*p_wV_m$
where $V_m: \cU_m \rightarrow
\cU$ is the inclusion map. Let $p_m$ denote the formal power series
\[
p_m = \displaystyle\sum_{w \in \Lambda} (p_m)_w w
\]
For each $X \in \cK$, Observe that $\|p_m(X)\| \leq \|p(X)\|$. Thus
$\|p_m\| \leq \|p\|$ and $p_m \in M_m(\cAK)$ for all $m \in \N$.
From Theorem {\ref{thm:matrixversion}}, there exists
$f_m \in M_m(\cIK)$ such that
$x_m = p_m + f_m \in M_m(\cAK)$ and
\[
\|x_m\| =
\sup\{\|p_m(X)\|: X \in \cK, X \mbox{ is } \Lambda-\mbox{nilpotent}\}.
\]

For $w \in \cFd$, define $B(\cU) \ni (\tilde{x}_m)_w = V_m (x_m)_w V_m^*$.
Let $\tilde{x}_m$ denote
the formal power series $\sum_{w \in \cFd} (\tilde{x}_m)_w w$.
For $X \in \cK$ and $j = 0, 1, 2,...$, it follows from From
Lemma {\ref {lem:A-bounded}} that there exists $0 \le \rho < 1$ such that

\begin{equation}
\label{eq:nthtermbounded}
 \begin{split}
 \| \sum_{|w| = j} (\tilde{x}_m)_w \otimes X^w \|
  & \le  \| \displaystyle \sum_{|w| = j} (x_m)_w \otimes X^w \| \\
  & \le \rho^j \|x_m\| \\
  &  \le C \rho^j
\end{split}
\end{equation}
This implies that series for $\tilde{x}_m(X)$ converges for each $X \in \cK$
 and moreover we have
 \begin{equation}
 \label{eq:bounded}
 \|\tilde{x}_m\| \le \|x_m\|  \le C.
 \end{equation}
Recall $\gamma$ and $S(\ell)$ from Subsection {\ref {subsec:genesis}}.
Let $u \in \cU$ be an arbitrary unit vector. For each $0 \leq j \le \ell$
and $X \in \cK$, it follows that

\begin{align*}
C^2 &\ge \| \displaystyle \sum_{|w| = j} (\tilde{x}_m)_w \otimes
S(\ell)^w (u \otimes \emptyset) \|^2 \\
    &\ge \| \displaystyle \sum_{|w| = j} (\tilde{x}_m)_wu \otimes
    w \|^2\\
    &\ge \sum_{|w|=j} \|(\tilde{x}_m)_w u\|^2
\end{align*}
Thus $\|(\tilde{x}_m)_w\| \le C$ for all $w \in \cFd$
and $m \in \N$.

Since $\cU$ is a separable Hilbert space and the sequence
$\{(\tilde{x}_m)_w\}_{m=1}^{\infty}$ is bounded (by C), for each
$w \in \cFd$, there exists a subsequence of
$\{(\tilde{x}_m)_w\}_{m=1}^{\infty}$ that
converges with respect to the WOT on $B(\cU)$. By a diagonal argument
similar to the one in Lemma {\ref{lem:subsequence}}, it follows that
there exists a subsequence $\{\tilde{x}_{m_k}\}$ of $\{\tilde{x}_m\}$
and $\{\tilde{x}_w\}_{w \in \cFd} \subset B(\cU)$ such that
for each $w \in \cFd$
\[
(\tilde{x}_{m_k})_w \rightarrow \tilde{x}_w
\]
with respect to the WOT on $B(\cU)$.

Let  $\tilde{x}$ denote the formal power series
$\sum_{w \in \cFd} \tilde{x}_w w$.
The proof of the Theorem is completed
by showing that $\| \tilde{x}\|\le C$
and that noting that
 $\tilde{x}_w = \lim (\tilde{x}_{m_k})_w = \lim V_{m_k}(p_{m_k})_wV_{m_k}^* = p_w$
(WOT limits) for $w\in\Lambda$.  The
details are omitted.
\end{proof}

\section{Examples and the case of infinite intial segments $\Lambda$}
\label{sec:further}
  Of course the results of this paper apply to the
  examples in Subsection \ref{subsec:examples}.

  In the case of the non-commutative matrix polydisc,
  the operators obtained by applying the representation
  of the quotient algebra to the generators $[g_j] = g_j + \cIK$;
  $1 \le j \le d$,
  are automatically contractions and thus certain technical
  details of the proof of Theorem \ref{thm:main}  are
  absent. Consequently, the argument easily extends
  to handle infinite intial segments, provided the
  underlying domain is expanded to include operators
  on separable Hilbert space.

 Fix a separable infinite dimensional Hilbert space $\cH$.
 Let $\cN^d$ denote the {\it operator non-commutative polydisc}
\[
\cN^d = \{(T_1, T_2, ..., T_d) \, :
 \, T_j \in B(\cH) \,\, \& \,\, \|T_j\| < 1 \}.
\]

The following variant of Theorem {\ref{thm:operatorversion}} holds.

\begin{theorem}
\label{thm:ncpolydisc}
Let $\cU$ be a separable Hilbert space,
 $\Lambda \subset \cFd$ be an infinite initial segment and
$p = \sum_{w \in \Lambda} p_w w$ be a formal
power series with coefficients $p_w \in B(\cU)$ such that
$\|p\| = \sup\{\|p(X)\| \, : \, X \in \cN^d \} < \infty.$
There exists operators $\tilde{x}_w \in B(\cU)$ and a formal
power series $\tilde{x} = \sum_{w \in \cFd} \tilde{x}_w w$
such that $\tilde{x}_w = p_w$ for all $w \in \Lambda$ and
$\|\tilde{x}\| = \sup \{\|p(T)\| \, : \, T \in \cN^d, \, T \mbox{ is }
\Lambda-\mbox{nilpotent}\}.$
\end{theorem}

Similarly, consider the $dd'$-dimensional
{\it operator non-commutative mixed ball},
\[
\cD^{dd'} = \{ T = (T_{11}, T_{12},...,T_{dd'}) \, : \, T_{ij} \in B(\cH)
\,\, \& \,\, \|T\|_{op} <1\}
\]
where $\|T\|_{op}$ is the norm of the operator
$(T_{ij})_{i,j=1}^{d,d'} : B(\cH^{d'}) \rightarrow B(\cH^d)$.

A variant of Theorem {\ref{thm:operatorversion}} holds in this case
as well, the statement of which can be obtained by replacing
$\cN^d$ in the statement of Theorem {\ref{thm:ncpolydisc}} by
$\cD^{dd'}$.

\section*{Acknowledgements}
I would like to thank my advisor Scott McCullough
for his guidance in the preparation of this article.


\begin{thebibliography}{99}
\bibitem[A]{A} J. Agler: On the representation of certain holomorphic functions defined
on a polydisc, Topics in operator theory: Ernst D. Hellinger memorial volume, volume 48 of
Oper. Theory Adv. Appl., pp 47-66. Birkh¨auser, Basel, 1990.

\bibitem[BB]{BB} Ball, Joseph A.; Bolotnikov, Vladimir: Interpolation in the
 noncommutative Schur-Agler class.  J. Operator Theory  58  (2007),  no. 1, 83--126.

\bibitem[BGM]{BGM} Ball, Joseph A.; Groenewald, Gilbert; Malakorn, Tanit: Conservative structured
 noncommutative multidimensional linear systems.
 The state space method generalizations and applications,  179--223, Oper. Theory Adv. Appl., 161,
 Birkhäuser, Basel, 2006.

\bibitem[BLTT]{BLTT} J.A. Ball; W.S. Li; D. Timotin; T. T. Trent: A commutant
liftint theorem on the polydisc, Indiana Univ. Math. J., 48(2):653-675, 1999.
\bibitem[C]{C} J. Cuntz: Simple C*-algebras generated by isometries,
Comm. Math. Phys. 57, 173-185 (1977).

\bibitem [CF] {CF} C. Carath$\acute{\text{e}}$odory; L. Fej$\acute{\text{e}}$r:
¨Uber den Zusammenhang der Extremen von harmonischen
Funktionen mit ihren Koeffizienten und ¨uber den Picard–Landau'schen Satz,
Rend. Circ. Mat. Palermo, 32: pp 218-239, 1911.

\bibitem[Co]{Co} T. Constantinescu; J. L. Johnson: A note on noncommutative
 interpolation. Canad. Math. Bull., 46(1):59–70, 2003.

\bibitem[D]{D} Sh.A. Dautov; G. Khuda$\breve{\text{i}}$berganov:
The Caratheodory.Fejer
problem in higher-dimensional complex
analysis, Sibirsk. Mat. Zh. 23 (2) (1982) 58.64, 215.

\bibitem[EW]{EW} Effros, Edward G.; Winkler, Soren Matrix convexity:
operator analogues of the bipolar and Hahn-Banach theorems,
J. Funct. Anal.  144  (1997),  no. 1, 117--152.

\bibitem[EPP]{EPP} Eschmeier, Jörg; Patton, Linda; Putinar, Mihai
: Carathéodory-Fejér interpolation on polydisks, Math. Res. Lett.  7  (2000),
no. 1, 25--34.

\bibitem[FF]{FF} C. Foias; A.E. Frazho: The commutant lifting approach
 to interpolation problems,
Operator Theory: Advances and Applications, vol. 44, Birkhäuser, Verlag, Basel, 1990.

\bibitem[HM]{HM} Helton, J. William; McCullough, Scott:
Every free basic semi-algebraic set has an LMI representation,
arXiv:0908.4352v2

\bibitem[HKMS]{HKMS} Helton, J. William; Klep, Igor; McCullough, Scott; Slinglend, Nick:
 Noncommutative ball maps,  J. Funct. Anal.  257  (2009),  no. 1, 47--87.

\bibitem[KV]{KV} D. Kalyuzhny$\breve{\text{i}}$-Verbovetzki$\breve{\text{i}}$:
Carath´eodory Interpolation on the Noncommutative Polydisk,
J. Funct. Anal., 229 (2005), pp. 241-276.

\bibitem[KVV]{KVV} D. Kalyuzhnyi-Verbovetski; V. Vinnikov: Foundations of
noncommutative function theory, in preparation.

\bibitem[P]{P} V. Paulsen: Completely Bounded Maps and Operator Algebras,
 Cambridge University Press, 1st edition, Jan 2003.

\bibitem[P1]{P1} G. Popescu:  Free holomorphic functions on the unit ball of  $B(H)^n$,
  J. Funct. Anal. 241 (2006), pp 268-333.

\bibitem[P2]{P2} G. Popescu: Free holomorphic functions and interpolation,
Math. Ann. 342 (2008) 1–30.

\bibitem[P3]{P3} G. Popescu: Interpolation problems in several variables,
J. Math. Anal. Appl., 227(1):227–250, 1998.

\bibitem[S]{S} D. Sarason: Generalized interpolation in $H^\infty$,
Trans. Amer. Math. Soc. 127 (1967), pp. 179-203.

\bibitem[Sc]{Sc} I. Schur: ¨Uber Potenzreihen die im Innern des E
inheitskreises beschr¨ankt sind,
J. Reine Angew. Math., 147: pp 205-232, 1917.

\bibitem[T]{T} O. Toeplitz: Über die Fourier'sche Entwickelung positiver
 Funktionen, Rend. Circ. Mat. Palermo 32 (1911) 191–192.

\bibitem[V1]{V1} D. V. Voiculescu: Free Probability Theory,
American Mathematical Society, 1997.

\bibitem[V2]{V2} D. V. Voiculescu, Free analysis questions I: Duality transform
 for the coalgebra of ÝX:B, Int. Math. Res. Not. 16
(2004) 793.822.

\bibitem[V3]{V3} D. V. Voiculescu, K. J. Dykema, A. Nica, Free Random Variables:
a noncommutative probability
 approach to free products with applications to random matrices, operator algebras,
 and harmonic analysis on free groups, American Mathematical Society, 1992.

\end{thebibliography}
\end{document}